\documentclass[hidelinks]{amsart}

\usepackage[defaultlines=2,all]{nowidow} 

\usepackage{ mathtools, amssymb, stmaryrd, mathdots, bbm, mleftright, faktor }

\usepackage{ letltxmacro }    
\usepackage{ xparse }
\usepackage{ stackengine, graphicx, scalerel, etoolbox }

\usepackage[usenames,dvipsnames,svgnames,table,xcdraw,rgb]{xcolor}
\definecolor{darkazure}{HTML}{0059b3}
\definecolor{azure}{HTML}{007fff}
\definecolor{paleazure}{HTML}{a6d2ff}
\definecolor{vdarkgold}{HTML}{807100}
\definecolor{darkergold}{HTML}{b39e00}
\definecolor{darkgold}{HTML}{f2d400}
\definecolor{gold}{HTML}{ffdf00}
\definecolor{palegold}{HTML}{fff7bf}
\definecolor{beaverred}{HTML}{cc0000}
\definecolor{palered}{HTML}{f68c67}
\newcommand{\colourA}{azure}

\newcommand{\colourB}{darkergold}

\newcommand{\colourC}{beaverred}

\usepackage[outline]{contour} 
\contourlength{.23pt}

\definecolor{dark-blue}{rgb}{0.15,0.15,0.4}

\usepackage{ tikz, tikz-cd }
\usetikzlibrary{shapes.misc, calc, decorations.pathreplacing, patterns, arrows.meta, positioning}

\usepackage{ booktabs }
\usepackage{siunitx}
\usepackage[shortlabels]{ enumitem }

\SetEnumitemKey{smallin}{leftmargin=2.7em, labelindent=*}   
\setlist{smallin, topsep=1pt} 

\makeatletter    
\newcommand{\prelistcommand}{\nobreak\leavevmode\@nobreaktrue}
\makeatother
\SetEnumitemKey{beginthm}{before=\prelistcommand}   

\usepackage{ standalone }
\usepackage[section]{ placeins } 

\counterwithin{figure}{section}
\counterwithin{table}{section}

\usepackage[margin=0.2cm]{ caption }
\usepackage[labelformat=simple]{ subcaption }

\makeatletter
\def\subsection{\@startsection{subsection}{2}%
  \z@{.5\linespacing\@plus.7\linespacing}{.3\linespacing}%
  {\normalfont\bfseries}}
\makeatother

\makeatletter
\def\subsubsection{\@startsection{subsubsection}{3}%
  \z@{5ex\@plus1ex}{-1ex}%
  {\normalfont\itshape}}%
\makeatother

\usepackage{ amsthm, thmtools, url }
\usepackage[colorlinks=true, linkcolor=black, citecolor=black, urlcolor=dark-blue, anchorcolor=dark-blue]{ hyperref }
\usepackage[nameinlink, capitalize]{ cleveref }

\crefname{definition}{Def.}{Defs.}
\Crefname{definition}{Definition}{Definitions}

\numberwithin{equation}{section}
\declaretheorem[style=plain,numberlike=equation]{theorem}
\declaretheorem[style=plain,numberlike=theorem]{lemma}
\declaretheorem[style=plain,numberlike=theorem]{proposition}
\declaretheorem[style=plain,numberlike=theorem]{corollary}

\declaretheorem[style=remark,numberlike=theorem]{remark}

\declaretheorem[style=definition,numberlike=theorem]{definition}

\crefformat{section}{\S#2#1#3}
\crefrangeformat{section}{\S#3#1#4--#5#2#6}
\crefmultiformat{section}{\S#2#1#3}{ and~\S#2#1#3}{, \S#2#1#3}{, and~\S#2#1#3}
\crefrangemultiformat{section}{\S#2#1#3}{ and~\S#2#1#3}{, \S#2#1#3}{, and~\S#2#1#3}

\crefformat{equation}{(#2#1#3)}
\crefrangeformat{equation}{(#3#1#4)--(#5#2#6)}
\crefmultiformat{equation}{(#2#1#3)}{ and~(#2#1#3)}{, (#2#1#3)}{, and~(#2#1#3)}
\crefrangemultiformat{equation}{(#2#1#3)}{ and~(#2#1#3)}{, (#2#1#3)}{, and~(#2#1#3)}

\crefname{enumi}{}{}

\newcommand{\initialsspace}{0.1em}
 
\makeatletter
\newcommand{\splitlist}[1]{\@splitlist#1\@nil}
\def\@splitlist#1\@nil{%
  \if\relax\detokenize{#1}\relax
    \expandafter\@gobble
  \else
    \expandafter\@firstofone
  \fi
  {\@spl@tlist#1.\@nil}%
}
 
\def\@spl@tlist#1.#2\@nil{%
    \def\tmpA{#1}%
    \def\tmpB{#2}%
    \def\tmpP{.}%
    \ifx\tmpB\tmpP%
        #1.%
    \else{%
        \ifx\tmpA\@empty%
        \else%
                #1.\nobreak\hspace{\initialsspace}%
        \fi%
    }%
    \fi%
  \if\relax\detokenize{#2}\relax
    \expandafter\@firstoftwo
  \else
    \expandafter\@secondoftwo
  \fi
  {\unskip}%
  {\@spl@tlist#2\@nil}%
}
\makeatother

\NewDocumentCommand\set{s m}{%
    \IfBooleanTF#1%
    {\left\{ #2 \right\}}%
    {\{#2\}}%
}
\NewDocumentCommand\setbuild{s m m}{%
    \IfBooleanTF#1%
    {\ensuremath{\left\{\, #2 \, \middle| \, #3 \,\right\}}}%
    {\ensuremath{\{\, #2 \, \mid \, #3 \,\}}}%
}

\makeatletter
\newcommand\notni{\mathrel{\m@th\mathpalette\canc@l\owns}}
\newcommand\canc@l[2]{{\ooalign{$\hfil#1/\mkern1mu\hfil$\crcr$#1#2$}}}
\makeatother

\DeclarePairedDelimiter{\abs}{\lvert}{\rvert}

\newcommand{\la}{\lambda}

\renewcommand{\epsilon}{\varepsilon}
\renewcommand{\phi}{\varphi}
\renewcommand{\leq}{\leqslant}

\renewcommand{\geq}{\geqslant}

\newcommand{\blank}{{-}}    

\makeatletter
\let\@@pmod\pmod
\DeclareRobustCommand{\pmod}{\@ifstar\@pmods\@@pmod}
\def\@pmods#1{\mkern4mu({\operator@font mod}\mkern 6mu#1)}
\makeatother

\usepackage{mfabacus_edited}

\newlength{\minarrow}
\NewDocumentCommand{\myarrow}{sm}{
  \IfBooleanTF{#1}{
    \xrightarrow{#2}
  }{
    \xrightarrow{\mathmakebox[\minarrow]{#2}}
  }
}

\tikzset{
arstermonab/.style={inner sep=0, anchor=east},
arsterm/.style={inner sep=1.7},
}

\newcommand{\Gcal}{\mathcal{G}}

\title
[Large \(p\)-core \(p'\)-partitions and walks on the additive residue graph]
{Large \(p\)-core \(p'\)-partitions and \\ walks on the additive residue graph}
\author{Eoghan McDowell}

\usepackage[foot]{ amsaddr }
\makeatletter
\patchcmd{\@setfoot@addresses}{\scshape\ignorespaces}{\ignorespaces}{}{} 
\makeatother

\address{%
\emph{Affiliation}: Okinawa Institute of Science and Technology.
}

\email{eoghan.mcdowell@oist.jp}

\makeatletter
\@namedef{subjclassname@2020}{\textup{2020} Mathematics Subject Classification}
\makeatother


\makeatletter
\def\@setthanks{\vspace{-\baselineskip}\def\thanks##1{\@par##1\@addpunct.}\thankses}
\makeatother

\thanks{
\vspace{3pt}
This is the accepted manuscript for an article published in Annals of Combinatorics, which is available online at \href{https://doi.org/10.1007/s00026-022-00622-2}{https://doi.org/10.1007/s00026-022-00622-2} or at \href{https://rdcu.be/c0sGO}{https://rdcu.be/c0sGO}%
}

\begin{document}

\begin{abstract}
This paper investigates partitions which have neither parts nor hook lengths divisible by \(p\), referred to as \(p\)-core \(p'\)-partitions.
We show that the largest \(p\)-core \(p'\)-partition corresponds to the longest walk on a graph with vertices \(\{0, 1, \ldots, p-1\}\) and labelled edges defined via addition modulo \(p\).
We also exhibit an explicit family of large \(p\)-core \(p'\)-partitions, giving a lower bound on the size of the largest such partition which is of the same degree as the upper bound found by McSpirit and Ono.
\end{abstract}

\maketitle

\section{Introduction}

A partition which has no hook lengths divisible by \(p\) is called \emph{\(p\)-core}, and a partition which has no parts divisible by \(p\) is here called a \emph{\(p'\)-partition} (this property is sometimes called being \(p\)-regular, though that terminology is often used to mean having no part repeated \(p\) or more times).
For \(p\) an odd prime, this paper investigates large \(p\)-core \(p'\)-partitions. 

McSpirit and Ono \cite[Theorem~4.1]{mcspiritono2022chartable} show that the size of a \(p\)-core \(p'\)-partition is bounded above by
\[
    \frac{1}{24}(p^6 - 2p^5 + 2p^4 - 3p^2 + 2p).
\]
We give a marginal improvement on this bound in \Cref{prop:upper_bound} (the leading term is unchanged; the degree 5 term has coefficient \(-\frac{1}{6}\) in place of \(-\frac{1}{12}\)).

We exhibit in \Cref{section:explicit} an example of a \(p\)-core \(p'\)-partition of size 
\[
\frac{1}{96}(p^6 +6p^4-24p^3+89p^2-120p-48).
\]
This serves as a lower bound on the size of the largest \(p\)-core \(p'\)-partition, and shows that McSpirit and Ono's upper bound is of optimal degree.

We introduce in \Cref{section:walks} the \emph{additive residue graph} as the labelled directed graph on \(\{0,1, \ldots, p{-}1\}\) with edges labelled \(i\) corresponding to addition of \(i\) modulo \(p\).
We prove that the largest \(p\)-core \(p'\)-partition corresponds to the longest walk on this graph which traverses edges in increasing order and which avoids \(0\).
The largest \(p\)-core \(p'\)-partitions and their sizes can be computed using this characterisation; we record this data, as well as the sizes of the bounds above, in \Cref{appendix}.

\subsection{Significance of \texorpdfstring{\(p\)}{p}-core \texorpdfstring{\(p'\)}{p'}-partitions for the symmetric group}

Partitions of \(n\) index both conjugacy classes and simple characters of the symmetric group \(S_n\).
A class labelled by a \(p'\)-partition is a \(p'\)-class (that is, consisting of elements of order not divisible by \(p\)).
A simple character labelled by a \(p\)-core partition remains simple upon reduction modulo \(p\) (though not conversely; more generally, two characters lie in the same \(p\)-block if and only if the labelling partitions have the same \(p\)-core, a result known as Nakayama's Conjecture \cite[6.1.21]{jameskerber1984reptheory}).

Motivation for knowing the sizes of \(p\)-core \(p'\)-partitions is provided by McSpirit and Ono \cite{mcspiritono2022chartable}.
The Murnaghan--Nakayama rule \cite[2.4.7]{jameskerber1984reptheory} implies that characters of \(S_n\) labelled by \(p\)-core partitions vanish away from \(p'\)-classes.
Thus if no \(p\)-core \(p'\)-partitions of a given size \(n\) exist, 
then the character table of \(S_n\) restricted to \(p\)-core partitions is all zeros.
McSpirit and Ono use this fact
to give asymptotics for the numbers of zeros in the restricted character table \cite[Corollary~1.4]{mcspiritono2022chartable}.
The results of the present paper -- a lower bound on the size \(N\) of the largest \(p\)-core \(p'\)-partition, and a characterisation of \(N\) in terms of a certain walk -- are therefore also descriptions of an \(N\) such that the restricted character table of \(S_n\) is guaranteed to be all zeros for all \(n > N\).

\section{Describing partitions}
\label{section:parameters}

We introduce three parameters that are convenient for describing \(p\)-core partitions.
In this section \(p\) is permitted to be any positive integer.
For a full introduction to abacus notation, see for example \cite[\S 2.7]{jameskerber1984reptheory}.

\subsection{Bead multiplicities}

We view partitions on the \(p\)-abacus: there are \(p\) vertical \emph{runners}, labelled \(0\) to \(p{-}1\), on which we place \emph{beads} encoding the partition.
Positions on the abacus are ordered left to right and top to bottom: for \(0 \leq i \leq p{-}1\) and \(j \geq 1\), the \(((j-1)p + i)\)th position on the abacus is on the \(i\)th runner in the \(j\)th row.
Each bead contributes a part of size equal to the number of gaps preceding that bead (that is, gaps with a numerically lower position).
In this paper, unless otherwise specified we assume our abaci have a gap in the \(0\)th position; with this assumption there is a one-to-one correspondence between partitions and abaci.

The \emph{length} of a partition, denoted \(\ell(\blank)\), is its number of parts.
Assuming there is a gap in the \(0\)th position, this equals the number of beads on the abacus. 

A partition is \(p\)-core if and only if all the beads are at the top of their runners \cite[proof of Theorem 2.7.16]{jameskerber1984reptheory}.
The following parameters therefore uniquely determine a \(p\)-core partition.

\begin{definition}
\label{def:bead_multiplicities}
For \(1 \leq i \leq p-1\), the \emph{\(i\)th bead multiplicity} of a \(p\)-core partition is the number of beads on the \(i\)th runner in the corresponding abacus (that is, the corresponding abacus which has a gap in the \(0\)th position).
\end{definition}

\begin{lemma}
\label{lemma:size_formula}
Let \(\la\) be a \(p\)-core partition with bead multiplicities \((b_1, \ldots, b_{p-1})\).
Then \(\ell(\la) = \sum_{i=1}^{p-1} b_i\) and
\[\abs{\la} =  \frac{1}{2}\ell(\la)(1-\ell(\la)-p) + \frac{p}{2} \sum_{i=1}^{p-1}b_i^2 + \sum_{i=1}^{p-1}ib_i.\]
\end{lemma}

\begin{proof}
Let \(B_k\) denote the position on the abacus of the bead corresponding to the \(k\)th part of \(\la\).
Then \(\la_k = B_k - (\ell(\la)-k)\) (because the size of the part is given by the number of gaps preceding the bead, which is the position of the bead minus the number of beads preceding it).
The expression follows by summing over all beads and breaking up the sum over runners.
\end{proof}

\subsection{Row multiplicities}

Given an abacus whose beads are not all rightmost in their rows, shifting all the beads to the right yields a larger partition.
Shifting beads to the right does not necessarily preserve the \(p'\)-partition property; nevertheless, we obtain the following.

\begin{lemma}
\label{lemma:rightmost}
The largest \(p\)-core \(p'\)-partition has all beads rightmost in their rows.
\end{lemma}

\begin{proof}
Suppose we have a \(p\)-core \(p'\)-partition whose beads are not all rightmost in their rows; we will show that there exists a larger \(p\)-core \(p'\)-partition.

Consider two beads in positions \(x,y\) such that \(x < y\) and there is no bead in position \(z\) for any \(x < z < y\).
If the bead at \(x\) is moved to position \(y-1\), this increases the size of the part corresponding to the moved bead without changing the size of any other part.
The new size of the part corresponding to the moved bead is precisely the size of the part corresponding to the bead in position \(y\); in particular, if no parts are of size divisible by \(p\) before the move, then none are after.

If the \((p{-}1)\)th runner has the (joint) most beads, we can therefore shift all the beads on the abacus as far right as possible within their rows to obtain a larger \(p\)-core \(p'\)-partition.
Otherwise, we consider a different abacus configuration for the same partition by adding a bead in the \(0\)th position and increasing the positions of all other beads by \(1\).
Equivalently, we are adding a bead to the \((p{-}1)\)th runner and moving it into the \(0\)th runner position, and shifting the other runners one place to the right.
We repeat until the (new) \((p{-}1)\)th runner does contain the (joint) most beads, permitting us to shift all the beads to the right within their rows.
\end{proof}

Given \Cref{lemma:rightmost}, we restrict our attention to abaci whose beads are rightmost within their rows. 
The following parameters uniquely determine such a partition.

\begin{definition}
\label{def:row_multiplicities}
For \(1 \leq i \leq p-1\), the \emph{\(i\)th row multiplicity} of a \(p\)-core partition whose beads are rightmost within their rows is the number of rows containing exactly \(i\) gaps and \(p-i\) beads.
\end{definition}

Equivalently, we are restricting our attention to abaci where the bead multiplicities are weakly increasing, and the \(i\)th row multiplicity is the difference between the \(i\)th and \((i{-}1)\)th bead multiplicities.

\subsection{Abacus residue sequence}

For a \(p\)-core partition whose beads are rightmost within their rows, all the beads in a row contribute parts of equal size.
That size is the number of gaps in the row plus the common size of the parts contributed by the beads in the preceding row.

\begin{definition}
\label{def:abacus residue sequence}
The \emph{abacus residue sequence} of a \(p\)-core partition whose beads are rightmost within their rows is the sequence of residues modulo \(p\) of the sizes of the parts contributed by each row.
\end{definition}

The abacus residue sequence can be obtained from the row multiplicities as follows:
start with the empty sequence; for each \(i\) (beginning with \(i=1\)),
append a term equal to \(i\) plus the previous term, 
and do so a number of times equal to the \(i\)th row multiplicity.
An abacus residue sequence uniquely determines the abacus for a \(p\)-core partition whose beads are rightmost within their rows.
A partition is a \(p'\)-partition if and only if \(0\) does not occur in the abacus residue sequence.

In \Cref{fig:abacus_parameters} we illustrate all the parameters introduced in this section.

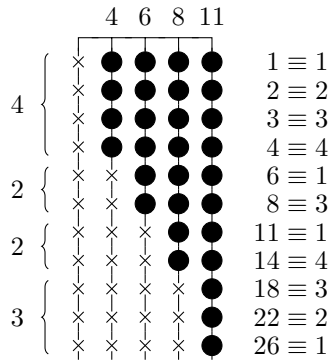
\begin{figure}[htb]
    \centering
    \makeatletter
\begin{tikzpicture}[scale=\abas*1.2,,baseline={([yshift=-.8ex]current bounding box.center)}]
\@bacus
lmmmr,xbbbb,xbbbb,xbbbb,xbbbb,xxbbb,xxbbb,xxxbb,xxxbb,xxxxb,xxxxb,xxxxb
.
\node[inner sep=0, at={(\abah,\abav+0.1cm)}] (b1) {\(4\)};
\node[inner sep=0, at={(\abah*2,\abav+0.1cm)}] (b2) {\(6\)};
\node[inner sep=0, at={(\abah*3,\abav+0.1cm)}] (b3) {\(8\)};
\node[inner sep=0, at={(\abah*4,\abav+0.1cm)}] (b4) {\(11\)};
\coordinate[at={(0cm,-\abav*0.5)}] (first1);
\coordinate[at={($(first1)+(0,-\abav*3)$)}] (last1);
    \abbracevl{first1}{last1}{4}
\coordinate[at={($(first1)+(0,-\abav*4)$)}] (first2);
\coordinate[at={($(first1)+(0,-\abav*5)$)}] (last2);
    \abbracevl{first2}{last2}{2}
\coordinate[at={($(first1)+(0,-\abav*6)$)}] (first3);
\coordinate[at={($(first1)+(0,-\abav*7)$)}] (last3);
    \abbracevl{first3}{last3}{2}
\coordinate[at={($(first1)+(0,-\abav*8)$)}] (first4);
\coordinate[at={($(first1)+(0,-\abav*10)$)}] (last4);
    \abbracevl{first4}{last4}{3}
\coordinate[at={(4cm,-\abav*0.5)}] (pos4);
\node[arstermonab, at={($(pos4)+(3.5cm,0cm)$)}] (1) {\(1 \equiv 1\)};
\node[arstermonab, at={($(1.0)+(0,-\abav)$)}] (2) {\(2 \equiv 2\)};
\node[arstermonab, at={($(2.0)+(0,-\abav)$)}] (3) {\(3 \equiv 3\)};
\node[arstermonab, at={($(3.0)+(0,-\abav)$)}] (4) {\(4 \equiv 4\)};
\node[arstermonab, at={($(4.0)+(0,-\abav)$)}] (6) {\(6 \equiv 1\)};
\node[arstermonab, at={($(6.0)+(0,-\abav)$)}] (8) {\(8 \equiv 3\)};
\node[arstermonab, at={($(8.0)+(0,-\abav)$)}] (11) {\(11 \equiv 1\)};
\node[arstermonab, at={($(11.0)+(0,-\abav)$)}] (14) {\(14 \equiv 4\)};
\node[arstermonab, at={($(14.0)+(0,-\abav)$)}] (18) {\(18 \equiv 3\)};
\node[arstermonab, at={($(18.0)+(0,-\abav)$)}] (22) {\(22 \equiv 2\)};
\node[arstermonab, at={($(22.0)+(0,-\abav)$)}] (26) {\(26 \equiv 1\)};
\end{tikzpicture}
\makeatother
    \caption{
    The \(5\)-abacus for \((26,22,18,14^2,11^2,8^3,6^3,4^4,3^4,2^4,1^4)\), a partition of \(198\).
    The bead multiplicities are indicated across the top, above the corresponding runners;
    the row multiplicities are indicated down the left-hand side;
    the abacus residue sequence is indicated down the right-hand side (the actual size of the parts corresponding to beads in a particular row are on the left-hand side of the equivalence).
    This partition is the largest \(5\)-core \(5'\)-partition.
    }
    \label{fig:abacus_parameters}
\end{figure}

In the terminology we have introduced, McSpirit and Ono obtain their upper bound on the size of a \(p\)-core \(p'\)-partition \cite[Theorem 4.1]{mcspiritono2022chartable} by observing that if some row multiplicity is \(p\) or greater, then the abacus residue sequence has a subsequence which is an arithmetic progression modulo \(p\) of length \(p\), and hence contains a \(0\) when \(p\) is prime.
In fact, when \(p\) is prime, an arithmetic progression modulo \(p\) which does not begin at \(0\) need only be of length \(p{-}1\) to guarantee a \(0\), and hence a partition is not a \(p'\)-partition if some row multiplicity other than the first nonzero multiplicity is \(p{-}1\) or greater.
This observation sharpens their upper bound to the following.

\begin{proposition}
\label{prop:upper_bound}
Let \(p\) be prime.
Let \(\la\) be a \(p\)-core \(p'\)-partition.
Then
\[
    \abs{\la} \leq \frac{1}{24}(p^6-4p^5+5p^4+12p^3-42p^2+52p-24).
\]
\end{proposition}

\begin{proof}
The first row multiplicity of the largest \(p\)-core \(p'\)-partition is nonzero:
if \((0,\ldots,0, m_i,\allowbreak\ldots, m_{p-1})\) are the row multiplicities of a \(p'\)-partition with \(m_i > 0\),
then \((i, 0, \ldots , 0, m_{i}-1, \ldots, m_{p-1})\) are the row multiplicities of a larger \(p'\)-partition.
Thus the requirement that the abacus residue sequence avoids long arithmetic progressions
-- described in the paragraph directly above --
implies that the largest \(p\)-core \(p'\)-partition has all row multiplicities less than or equal to \(p{-}2\), except possibly the first which is less than or equal to \(p{-}1\).

Now observe that increasing a row multiplicity increases the size of the partition.
Therefore the partition with row multiplicities \((p{-}1, p{-}2, p{-}2,\allowbreak \ldots,\allowbreak p{-}2)\) is larger than the largest \(p\)-core \(p'\)-partition.
Computing the size of this partition with \Cref{lemma:size_formula} (setting \(b_i = (p-2)i+1\)) gives the bound.
\end{proof}

\begin{remark}
The proof of McSpirit and Ono's bound \cite[Theorem~4.1]{mcspiritono2022chartable} erroneously claims that adding a bead to an abacus increases the size of the partition
(in their notation, that \(\mathfrak{A}_{\la} \leq \mathfrak{A}_{\Lambda}\) implies \(\abs{\la} \leq \abs{\Lambda}\)).
The above proof of \Cref{prop:upper_bound} corrects this mistake by considering the insertion of a row between existing rows (that is, increasing a row multiplicity) rather than the addition of a bead.
\end{remark}

\section{An explicit family of large \texorpdfstring{\(p\)}{p}-core \texorpdfstring{\(p'\)}{p'}-partitions}
\label{section:explicit}

We exhibit a family of large (having size sextic in \(p\)) \(p\)-core \(p'\)-partitions.

\begin{proposition}
\label{prop:lower_bound}
Let \(p\) be an odd prime.
Let \(\la\) be the \(p\)-core partition with row multiplicities
\((p{-}1,\, 2,\, p{-}2,\, 2,\, p{-}2,\, \ldots\,,\, 2,\, p{-}2,\, 1)\).
Then \(\la\) is a \(p'\)-partition of size \[
\abs{\la} = \frac{1}{96}(p^6 +6p^4-24p^3+89p^2-120p-48).
\]
\end{proposition}

\newcommand{\jumpbelow}{0.45}
\newcommand{\jumphalfheight}{\jumpbelow-0.32}

\newcommand{\jumparrow}[2]{
    \coordinate[at={($(#1)+(-0.14,\jumphalfheight)$)}] (start);
    \coordinate[at={($(#1)+(0.14,\jumphalfheight)$)}] (end);
    \draw[-{Stealth[width=1.3mm,length=0.8mm]}] (start)    .. controls
        ++(-75:0.25)
        and
        ++(-105:0.25) .. (end) node [midway, below] {\(\scriptstyle+#2\)};
}

\newcommand{\tikzbraceh}[3]{
    \draw [
    decoration={
        brace,
        mirror,
    },
    decorate
] ($(#1)+(-0.17,-0.5)$) -- ($(#2)+(0.17,-0.5)$) node [pos=0.5,anchor=north,yshift=-0.1cm] {\(#3\)};
}

\begin{proof}
The abacus residue sequence for \(\la\) is
\begin{center}
\begin{tikzpicture}[node distance=0.02cm]
\node[arsterm] (start1) at (0,0) {\(\big(1,\)};
\node[arsterm] (next1) [right=of start1] {\(2,\)};
\node[arsterm] (pen1) [right=of next1] {\(\!\!\vphantom{1}\ldots,\)};
\node[arsterm] (end1) [right=of pen1] {\(p{-}1,\)};
    \coordinate[left=of start1] (end0);
    \coordinate[at={($($($(end0)+(0.2,0)$)!1/2!(start1.180)$)+(0,-\jumpbelow)$)}] (truefirst1jump);
    \coordinate[at={($($(start1.0)!1/2!(next1.180)$)+(0,-\jumpbelow)$)}] (first1jump);
    \coordinate[at={($($(pen1.0)!1/2!(end1.180)$)+(0,-\jumpbelow)$)}] (last1jump);
    \jumparrow{truefirst1jump}{1}
    \jumparrow{first1jump}{1}
    \jumparrow{last1jump}{1}
    \node[at={($(first1jump)!1/2!(last1jump)$)}] (jump1midpoint) {\(\ldots\)};
    \tikzbraceh{first1jump}{last1jump}{p-2}{0,-0.5}
\node[arsterm] (next2) [right=of end1] {\(1,\)};
\node[arsterm] (end2) [right=of next2] {\(3,\)};
    \coordinate[at={($($(end1.0)!1/2!(next2.180)$)+(0,-\jumpbelow)$)}] (first2jump);
    \coordinate[at={($($(next2.0)!1/2!(end2.180)$)+(0,-\jumpbelow)$)}] (last2jump);
    \jumparrow{first2jump}{2}
    \jumparrow{last2jump}{2}
\node[arsterm] (next3) [right=of end2] {\(6,\)};
\node[arsterm] (pen3) [right=of next3] {\(\!\!\vphantom{1}\ldots,\)};
\node[arsterm] (end3) [right=of pen3] {\(p{-}3,\)};
    \coordinate[at={($($(end2.0)!1/2!(next3.180)$)+(0,-\jumpbelow)$)}] (first3jump);
    \coordinate[at={($($(pen3.0)!1/2!(end3.180)$)+(0,-\jumpbelow)$)}] (last3jump);
    \jumparrow{first3jump}{3}
    \jumparrow{last3jump}{3}
    \node[at={($(first3jump)!1/2!(last3jump)$)}] (jump3midpoint) {\(\ldots\)};
    \tikzbraceh{first3jump}{last3jump}{p-2}
\node[arsterm] (next4) [right=of end3] {\(1,\)};
\node[arsterm] (end4) [right=of next4] {\(5,\)};
    \coordinate[at={($($(end3.0)!1/2!(next4.180)$)+(0,-\jumpbelow)$)}] (first4jump);
    \coordinate[at={($($(next4.0)!1/2!(end4.180)$)+(0,-\jumpbelow)$)}] (last4jump);
    \jumparrow{first4jump}{4}
    \jumparrow{last4jump}{4}
\node[arsterm] (longdots) [right=of end4] {\(\vphantom{1}\ldots\ldots,\,\)};
\node[arsterm] (endm3) [right=of longdots] {\(p{-}2,\)};
\node[arsterm] (penm2) [right=of endm3] {\(\!\vphantom{1}\ldots,\,\)};
\node[arsterm] (endm2) [right=of penm2] {\(p{-}(p{-}2),\)};
    \coordinate[at={($($($(endm3.0)!1/2!(penm2.180)$)+(0,-\jumpbelow)$)+(-0.15,0)$)}] (firstm2jump);
    \coordinate[at={($($($(penm2.0)!1/2!(endm2.180)$)+(0,-\jumpbelow)$)+(0.15,0)$)}] (lastm2jump);
    \jumparrow{firstm2jump}{p{-}2}
    \jumparrow{lastm2jump}{p{-}2}
    \node[at={($(firstm2jump)!1/2!(lastm2jump)$)}] (jumpm2midpoint) {\(\ldots\)};
    \tikzbraceh{firstm2jump}{lastm2jump}{p-2}
\node[arsterm] (endm1) [right=of endm2] {\(1\big).\)};
    \coordinate[at={($($(endm2.0)!1/2!(endm1.180)$)+(0,-\jumpbelow)$)}] (m1jump);
    \jumparrow{m1jump}{p{-}1}
\end{tikzpicture}
\end{center}
This does not contain any zeros: when adding an even number \(2k\), the resulting residue is \(1\) or \(2k+1\); when adding an odd number \(2k+1\), the sequence starts at \(2k+1\), so \(2k+1\) can be added \(p{-}2\) times without hitting \(0\).
Thus \(\la\) is a \(p'\)-partition.

The claimed size is found using \Cref{lemma:size_formula}. (The bead multiplicities are
\[
    b_i = \begin{cases}
        \frac{1}{2}(i+1)p -1 & \text{if \(i\) odd;} \\
        \frac{1}{2}ip + 1 & \text{if \(i\) even and \(i < p-1\);} \\
        \frac{1}{2}p(p-1) & \text{if \(i=p-1\);}
    \end{cases}
\]
and the computations are eased by summing in pairs: for \(i\) odd, \(b_i + b_{p-i} = \frac{1}{2}p(p+1)\) except when \(i=1\), and \(b_{i}^2 + b_{i+1}^2 = \frac{1}{2}(i{+}1)^2p^2 + 2\) except when \(i=p{-}2\).)
\end{proof}

Except when \(p=3\), the \(p\)-core \(p'\)-partitions identified in \Cref{prop:lower_bound} are not largest.
Indeed, they do not satisfy the necessary conditions identified in \Cref{prop:walk_hits_p-1_with_every_i}(a) or \Cref{cor:row_multiplicity_symmetry} below.

We obtain two more families of large \(p\)-core \(p'\)-partitions by ``symmetrising'' the row multiplicities given in the statement of \Cref{prop:lower_bound}: replacing the second half of the tuple with the reverse of the first half, or vice versa (and adjusting the first or last entry by \(1\)).
Indeed, this corresponds to the same operation on the abacus residue sequence and hence preserves the \(p'\)-partition property (cf.~\Cref{cor:row_multiplicity_symmetry}).
The row multiplicities and sizes of the resulting partitions depend on the value of \(p\) modulo \(4\).
Which symmetrisation is larger also depends on the value of \(p\) modulo \(4\):
it can be shown that, for \(p > 3\),
replacing the first half with the reverse of the second yields the largest of the three partitions if \(p \equiv 1 \pmod{4}\),
while
replacing the second half with the reverse of the first yields the largest if \(p \equiv 3 \pmod{4}\).
Nevertheless neither offers a significant improvement: all are of size sextic in \(p\) with leading term \(\frac{1}{96} p^6\).

\section{Walks on the additive residue graph}
\label{section:walks}

We establish a correspondence between \(p\)-core \(p'\)-partitions whose beads are rightmost and walks on a certain graph, and show that the largest such partition corresponds to the longest such walk.

\subsection{Additive residue graph}

\begin{definition}
Let the \emph{additive residue graph} be the labelled directed graph \(\Gcal_p\) with vertices the residue classes modulo \(p\), and directed edges labelled \(i\) from each residue \(r\) to \(r+i\), for each \(1 \leq i \leq p{-}1\).
\end{definition}

As an unlabelled graph, \(\Gcal_p\) is the complete directed graph on \(p\) nodes.
We refer to an edge labelled \(i\) as an \(i\)-edge.
The graph \(\Gcal_5\) is illustrated in \Cref{fig:Gcal_5}.

\tikzset{
vertex/.style={circle,draw, minimum size=4.5mm, inner sep=0},
}
\usetikzlibrary{shapes.geometric}

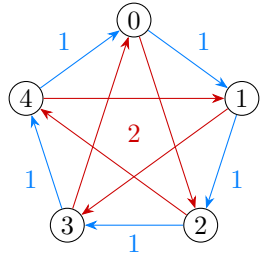
\begin{figure}[bt]
    \centering
\begin{tikzpicture}
\node[xscale=-1, regular polygon,regular polygon sides=5,minimum size=3cm] (pent) {};
    \node[vertex] (0) at (pent.corner 1) {\(0\)};
    \node[vertex] (1) at (pent.corner 2) {\(1\)};
    \node[vertex] (2) at (pent.corner 3) {\(2\)};
    \node[vertex] (3) at (pent.corner 4) {\(3\)};
    \node[vertex] (4) at (pent.corner 5) {\(4\)};
    \draw[-Stealth, color=\colourA] (0) -- (1) node [midway, auto, color=\colourA] {\(1\)};
    \draw[-Stealth, color=\colourA] (1) -- (2) node [midway, auto, color=\colourA] {\(1\)};
    \draw[-Stealth, color=\colourA] (2) -- (3) node [midway, auto, color=\colourA] {\(1\)};
    \draw[-Stealth, color=\colourA] (3) -- (4) node [midway, auto, color=\colourA] {\(1\)};
    \draw[-Stealth, color=\colourA] (4) -- (0) node [midway, auto, color=\colourA] {\(1\)};
    \draw[-Stealth, color=\colourC] (0) -- (2) node {};
    \draw[-Stealth, color=\colourC] (1) -- (3) node {};
    \draw[-Stealth, color=\colourC] (2) -- (4) node {};
    \draw[-Stealth, color=\colourC] (3) -- (0) node {};
    \draw[-Stealth, color=\colourC] (4) -- (1) node {};
    \node[color=\colourC] (centre) at (pent.center) {\(2\)};
\end{tikzpicture}
    \caption{The graph \(\Gcal_5\).
    For readability we have drawn only half the edges;
    reversing the edges labelled \(1\) and \(2\) yields the edges labelled \(4\) and \(3\) respectively.
    }
    \label{fig:Gcal_5}
\end{figure}

Consider the abacus residue sequence of a \(p\)-core partition whose beads are rightmost.
This sequence can be interpreted as the sequence of vertices visited in a walk on \(\Gcal_p\) starting at \(0\) in which edges appear in increasing order
(that is, some number of \(1\)-edges are traversed, then some number of \(2\)-edges, and so on).
The \(i\)th row multiplicity of the partition is the number of \(i\)-edges traversed; the total length of the walk is the \((p{-}1)\)th bead multiplicity.
The partition is a \(p'\)-partition if and only if the walk does not revisit \(0\).
As an example, the walk on \(\Gcal_5\) corresponding to the \(5\)-abacus of a partition is illustrated in \Cref{fig:walk_on_Gcal_5}.

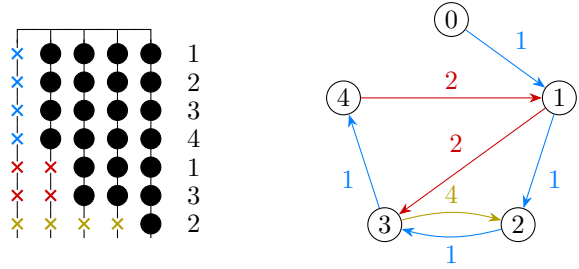
\begin{figure}[bt]
    \centering
    \makeatletter
\setcounter{grad}{120}
\begin{tikzpicture}[bend angle=0, scale=\abas*1.2,,baseline={([yshift=-.8ex]current bounding box.center)}]
\@bacus
lmmmr,Xbbbb,Xbbbb,Xbbbb,Xbbbb,ZZbbb,ZZbbb,%
YYYYb
.
\coordinate[at={(4cm,-\abav*0.5)}] (pos4);
\node[arstermonab, at={($(pos4)+(1.5cm,0cm)$)}] (1) {\(1\)};
\node[arstermonab, at={($(1.0)+(0,-\abav)$)}] (2) {\(2\)};
\node[arstermonab, at={($(2.0)+(0,-\abav)$)}] (3) {\(3\)};
\node[arstermonab, at={($(3.0)+(0,-\abav)$)}] (4) {\(4\)};
\node[arstermonab, at={($(4.0)+(0,-\abav)$)}] (6) {\(1\)};
\node[arstermonab, at={($(6.0)+(0,-\abav)$)}] (8) {\(3\)};
\node[arstermonab, at={($(8.0)+(0,-\abav)$)}] (12) {\(2\)};
\begin{scope}[shift={(13,-2.8)}]
\node[xscale=-1, regular polygon,regular polygon sides=5,minimum size=3cm] (pent) {};
    \node[vertex] (0) at (pent.corner 1) {\(0\)};
    \node[vertex] (1) at (pent.corner 2) {\(1\)};
    \node[vertex] (2) at (pent.corner 3) {\(2\)};
    \node[vertex] (3) at (pent.corner 4) {\(3\)};
    \node[vertex] (4) at (pent.corner 5) {\(4\)};
\path[-Stealth]
    (0) edge [bend left, color=\colourA] node [midway, auto, color=\colourA] {\(1\)} (1);
\path[-Stealth]
    (1) edge [bend left, color=\colourA] node [midway, auto, color=\colourA] {\(1\)} (2);
\path[-Stealth]
    (2) edge [bend left=15, color=\colourA] node [midway, auto, color=\colourA] {\(1\)} (3);
\path[-Stealth]
    (3) edge [bend left, color=\colourA] node [midway, auto, color=\colourA] {\(1\)} (4);
\path[-Stealth]
    (4) edge [bend left, color=\colourC] node [midway, auto, color=\colourC] {\(2\)} (1);
\path[-Stealth]
    (1) edge [bend left, color=\colourC] node [midway, auto, swap, color=\colourC] {\(2\)} (3);
\path[-Stealth]
    (3) edge [bend left=15, color=\colourB] node [midway, auto, color=\colourB] {\(4\)} (2);
\end{scope}
\end{tikzpicture}
\makeatother
    \caption{
    The \(5\)-abacus for the partition \((12, 8^3, 6^3, 4^4, 3^4, 2^4, 1^4)\) and the corresponding walk on \(\Gcal_5\).
    The abacus residue sequence, indicated down the right-hand side of the abacus, is the sequence of vertices visited by the walk.
    }
    \label{fig:walk_on_Gcal_5}
\end{figure}

From now on, we are interested only in walks starting at \(0\) in which edges appear in increasing order and which do not revisit \(0\), and we will refer to such restricted walks simply as \emph{walks}.
There is a one-to-one correspondence between walks on \(\Gcal_p\) and \(p\)-core \(p'\)-partitions whose beads are rightmost.

Our aim in this section is to describe the longest walk on \(\Gcal_p\) and show that it corresponds to the largest \(p\)-core \(p'\)-partition.
Note we have not yet justified the use of the word ``the'' in the previous sentence -- that is, we have not yet shown that there is a unique longest walk and a unique largest partition.

Observe that extending a given walk on \(\Gcal_p\) corresponds to a larger partition; however, not every longer walk corresponds to a larger partition.
For example, traversing two \(1\)-edges corresponds to the partition \((2^{p-1},1^{p-1})\), while traversing a single \(\frac{p-1}{2}\)-edge corresponds to \((\frac{p-1}{2}^{\scriptscriptstyle\frac{p+1}{2}})\); the latter is larger for \(p>11\).

\subsection{The \texorpdfstring{\((p{-}1)\)}{(p-1)}-recurrent property}

In this subsection we show that the following property is satisfied both by longest walks on \(\Gcal_p\) and by the walks on \(\Gcal_p\) corresponding to largest \(p\)-core \(p'\)-partitions.

\begin{definition}
\label{def:p-1-revisiting}
A walk on \(\Gcal_p\) is said to be \emph{\((p{-}1)\)-recurrent} if it contains an \(i\)-edge incident to \(p{-}1\) for every \(1 \leq i \leq p{-}1\). 
\end{definition}

\begin{samepage}
\begin{lemma}
\label{lemma:walk_uses_all_labels}
Let \(1 \leq i \leq p{-}1\).
\begin{enumerate}[(a)]
    \item\label{item:all_labels_partition}
The \(i\)th row multiplicity of a largest \(p\)-core \(p'\)-partition is nonzero.
    \item\label{item:all_labels_walk}
A longest walk on \(\Gcal_p\) contains an \(i\)-edge.
\end{enumerate}
\end{lemma}
\end{samepage}

\begin{proof}
The case \(i=1\) of part \Cref{item:all_labels_partition} was noted in the proof of \Cref{prop:upper_bound}; considering the corresponding walk on \(\Gcal_p\) also deals with the case \(i=1\) of part \Cref{item:all_labels_walk}.

We will show that, given a walk on \(\Gcal_p\) which does not contain an \(i\)-edge for some \(i > 1\) (but which does contain a \(1\)-edge), there exists a walk on \(\Gcal_p\) strictly containing it.
Clearly this suffices to prove part \Cref{item:all_labels_walk}.
It also proves part \Cref{item:all_labels_partition}: given a partition whose \(i\)th row multiplicity is \(0\), the corresponding walk on \(\Gcal_p\) does not contain an \(i\)-edge, and a walk strictly containing it corresponds to a larger partition.

Suppose we have a walk on \(\Gcal_p\) which does not contain an \(i\)-edge for some \(i\), and suppose \(i > 1\) is minimal with this property.
Let \(r\) be the endpoint of the final \((i{-}1)\)-edge in the walk. 

If \(i=p{-}1\),
we can add a \((p{-}2)\)-edge or \((p{-}1)\)-edge (depending on the value of \(r\)) to extend the walk.
If \(i < p{-}1\),
it suffices to find a nonempty walk from \(r\) to itself which uses only edges with labels \(i{-}1\), \(i\) and \(i{+}1\).
Supposing this is done, we can add this closed walk to our original and obtain a walk strictly containing it.

We use the following:
take \((i{-}1)\)-edges until reaching \(p-(i{-}1)\) (which is possible without hitting \(0\) because \(p-(i{-}1)\) is the final residue in the sequence \((i{-}1, 2(i{-}1), \ldots)\));
take two \(i\)-edges (so the walk reaches \(i{+}1\));
take \((i{+}1)\)-edges until reaching \(r\) (which is possible without hitting \(0\) because \(i{+}1\) is the first residue in the sequence \((i{+}1, 2(i{+}1), \ldots)\)).
That is, the closed walk we add is:
\[
    r
    \ \xrightarrow{i-1} \ 
    \cdots
    \ \xrightarrow{i-1} \ 
    p{-}(i{-}1)
    \ \myarrow{i} \ 
    1
    \ \myarrow{i} \ 
    i{+}1
    \ \xrightarrow{i+1} \ 
    \cdots
    \ \xrightarrow{i+1} \ 
    r. \qedhere
\]
\end{proof}

\begin{lemma}
\label{lemma:walk_uses_s+1_if_possible}
Let \(1 \leq i \leq p{-}1\).
Let \(s\) be the endpoint of the final \(i\)-edge in a longest walk on \(\Gcal_p\).
Then \(s{+}1\) cannot be reached from \(s\) on \(i\)-edges while avoiding \(0\).
\end{lemma}

\begin{proof}
Suppose, towards a contradiction, \(s{+}1\) can be reached from \(s\) in \(\Gcal_p\) using only \(i\)-edges and avoiding \(0\).
Then we can replace the step \(s \xrightarrow{i+1} s{+}i{+}1\) with a path \(s \xrightarrow{i} s{+}i \xrightarrow{i} \cdots \xrightarrow{i} s{+}1 \xrightarrow{i} s{+}i{+}1\) which is strictly longer.
\end{proof}

\begin{proposition}
\label{prop:walk_hits_p-1_with_every_i}
Let \(1 \leq i \leq p{-}1\).
\begin{enumerate}[(a)]
    \item\label{item:p-1-recurrent_partition}
A largest \(p\)-core \(p'\)-partition has a part of size \(p{-}1 \pmod{p}\) which differs in size from an adjacent part by \(i\).
    \item\label{item:p-1-recurrent_walk}
A longest walk on \(\Gcal_p\) contains an \(i\)-edge incident to \(p{-}1\).
\end{enumerate}
That is, walks corresponding to largest \(p\)-core \(p'\)-partitions and longest walks on \(\Gcal_p\) are \((p{-}1)\)-recurrent.
\end{proposition}

\begin{proof}{}
[\Cref{item:p-1-recurrent_partition}]
Suppose we have a partition with no part of size \(p{-}1 \pmod{p}\) which differs in size from an adjacent part by \(i\).
Consider the section of the abacus containing the rows with \(i\) gaps:
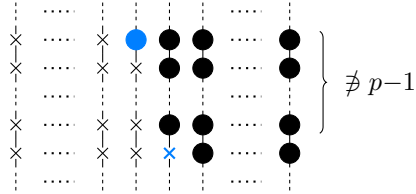
\begin{center}
\makeatletter
\begin{tikzpicture}[scale=\abas*1.2,,baseline={([yshift=-.8ex]current bounding box.center)}]
\@bacus
vhvvvvhv,%
xhxAbbhb,
xhxxbbhb,
vhvvvvhv,%
xhxxbbhb,
xhxxXbhb,
vhvvvvhv
.
\coordinate[overlay, at={(\abahd*2+\abah*5,-\abav*0.5)}] (posright);
\coordinate[overlay, at={(\abahd*2+\abah*5,-\abav*3.5)}] (posrightdown);
\abbracevr{posright}{posrightdown}{\notni p{-}1}
\end{tikzpicture}
\makeatother
\end{center}
The coloured bead can be moved to the position of the coloured gap.
This increases by \(1\) the parts in the rows in between, and leaves unchanged all other parts.
Since the residues of the rows in between are not \(p{-}1\), the result is still a \(p'\)-partition. 

[\Cref{item:p-1-recurrent_walk}]
Consider a walk on \(\Gcal_p\) with no \(i\)-edge incident to \(p{-}1\).
Let \(r\) be the startpoint of the first \(i\)-edge and \(s\) the endpoint of the last.
Thus the section of the walk on \(i\)-edges is
\[
    \cdots
    \ \myarrow{} \ 
    r{-}(i{-}1) 
    \ \xrightarrow{i-1} \
    r 
    \ \myarrow{i} \ 
    r{+}i 
    \ \myarrow{i} \ 
    \cdots 
    \ \myarrow{i} \ 
    s
    \ \xrightarrow{i+1} \ 
    \cdots.
\]
By \Cref{lemma:walk_uses_s+1_if_possible}, we may assume \(s{+}1\) cannot be reached from \(s\) in \(\Gcal_p\) on \(i\)-edges while avoiding \(0\) (or else the walk is not longest).
That is, \(s{+}1\) appears before \(s\) in the sequence \((i, 2i, \ldots)\), so \(s\) can be reached from \(s{+}1\) on \(i\)-edges while avoiding \(0\).
We can therefore replace the section above with the strictly longer
\[
    \cdots 
    \, \myarrow{} \,
    r{-}(i{-}1)
    \, \myarrow{i} \,
    r{+}1
    \, \myarrow{i} \,
    r{+}i{+}1
    \, \myarrow{i} \,
    \cdots
    \, \myarrow{i} \,
    s{+}1
    \, \myarrow{i} \,
    \cdots
    \, \myarrow{i} \,
    s
    \, \xrightarrow{i+1} \,
    \cdots.
\]
Since no \(i\)-edge in the original walk is incident to \(p{-}1\), the new walk still avoids \(0\).
(The new walk in fact corresponds to a larger partition; together with an analogue of \Cref{lemma:walk_uses_s+1_if_possible} for partitions, this gives an alternative proof of part \Cref{item:p-1-recurrent_partition}.)
\end{proof}

\subsection{Longest walks and largest partitions}

We now show that there is a unique longest walk on \(\Gcal_p\) and a unique largest \(p\)-core \(p'\)-partition, and that they correspond.

The \((p{-}1)\)-recurrent property allows a walk to be decomposed into closed walks from \(p{-}1\) to itself using only two types of edges.
Concretely, a \((p{-}1)\)-recurrent walk on \(\Gcal_p\) corresponds to a choice of residues \(1 \leq r_i \leq p{-}1\) on which the walk transitions from traversing \(i\)-edges to traversing \((i{+}1)\)-edges, for each \(2 \leq i \leq p{-}3\), as well as a choice of final residue \(1 \leq r_{p-1} \leq p-1\).
Once these choices are made, the walk is determined:
\begin{enumerate}[(i)]
    \item
take steps from \(0\) to \(p{-}1\) on \(1\)-edges;
    \item
for each \(2 \leq i \leq p{-}3\), take steps from \(p{-}1\) to \(r_i\) on \(i\)-edges and take steps from \(r_i\) to \(p{-}1\) on \((i{+}1)\)-edges;
    \item
take steps from \(p{-}1\) to \(r_{p-1}\) on \((p{-}1)\)-edges.
\end{enumerate}
Furthermore, although not all choices of a residue \(r_i\) yield a walk which avoids \(0\), this is independent of the choices of all other residues \(r_j\).

\begin{theorem}
\label{thm:longest_walk}
Let \(p\) be an odd prime.
The unique longest walk on \(\Gcal_p\) is the unique longest \((p{-}1)\)-recurrent walk.
\end{theorem}

\begin{proof}
By \Cref{prop:walk_hits_p-1_with_every_i}\Cref{item:p-1-recurrent_walk}, a longest walk on \(\Gcal_p\) is \((p{-}1)\)-recurrent, so it suffices to show that there is a unique longest \((p{-}1)\)-recurrent walk.
In a walk from \(p{-}1\) to itself using only \(i\)-edges and \((i{+}1)\)-edges,
the number \(x\) of \(i\)-edges and the number \(y\) of \((i{+}1)\)-edges satisfy \(xi + y(i{+}1) \equiv 0 \pmod{p}\) (because the walk starts and finishes on the same residue).
For a fixed total number of steps \(x+y\), there is at most one solution to this equation for \(0 \leq x,y \leq p-1\).
Thus there is a unique choice of residue \(r_i\) on which the walk transitions from \(i\)-edges to \((i{+}1)\)-edges which maximises the number of steps between two visits to \(p{-}1\).
\end{proof}

To show that a largest \(p\)-core \(p'\)-partition corresponds to the longest walk, we must show that maximising the number of steps between visits to \(p{-}1\) also maximises the size of the corresponding partition.

\begin{lemma}
\label{lemma:replacing_rows}
Let \(1 \leq i \leq p{-}1\).
Let \(0 \leq u < v \leq p{-}1\) be integers such that \(vi \geq u(i{+}1)\).
Consider an abacus which has all beads rightmost.
The following replacements strictly increase the size of the corresponding partition:
\begin{enumerate}[(a)]
    \item\label{item:replacing_gaps}
\(u\) consecutive rows each with \(i{+}1\) gaps becoming \(v\) rows each with \(i\) gaps;
    \item\label{item:replacing_beads}
\(u\) consecutive rows each with \(i{+}1\) beads becoming \(v\) rows each with \(i\) beads.
\end{enumerate}
\end{lemma}

The replacements \Cref{item:replacing_gaps} and \Cref{item:replacing_beads} are illustrated in \Cref{fig:abaci_replacements}.

\begin{figure}[htb]
\captionsetup[subfigure]{justification=centering}
\newcommand{\subfigurehspace}{\textwidth}
\setlength{\belowcaptionskip}{0pt}
\centering
\begin{subfigure}{\subfigurehspace}
    \centering
\makeatletter
\begin{tikzpicture}[scale=\abas*1.2,,baseline={([yshift=-.8ex]current bounding box.center)}]
\@bacus
vvvvvvv,%
xxxbbbb,%
xxxbbbb,%
Xxxbbbb,%
vvvvvvv
.
\node[at={(7.5,-1.4)}] (arrow) {\(\rightsquigarrow\)};
\abaxorig=9cm \abax=9cm \abay=0.85cm
\@bacus
vvvvvvv,%
xxbbbbb,%
xxbbbbb,%
xxbbbbb,%
xXbbbbb,%
xxbbbbb,%
vvvvvvv
.
\end{tikzpicture}
\makeatother
\caption{
\(u\) rows of \(i{+}1\) gaps \(\rightsquigarrow\) \(v\) rows of \(i\) gaps \\
The \(3\)rd-to-last gap is coloured; after the replacement, there are more beads following it.
}
\end{subfigure}
\par\bigskip
\begin{subfigure}{\subfigurehspace}
    \centering
\makeatletter
\begin{tikzpicture}[scale=\abas*1.2,,baseline={([yshift=-.8ex]current bounding box.center)}]
\@bacus
vvvvvvv,%
xxxxbbA,%
xxxxbbb,%
xxxxbbb,%
vvvvvvv
.
\node[at={(7.5,-1.4)}] (arrow) {\(\rightsquigarrow\)};
\abaxorig=9cm \abax=9cm \abay=0.85cm
\@bacus
vvvvvvv,%
xxxxxbb,%
xxxxxAb,%
xxxxxbb,%
xxxxxbb,%
xxxxxbb,%
vvvvvvv
.
\end{tikzpicture}
\makeatother
\caption{
\(u\) rows of \(i{+}1\) beads \(\rightsquigarrow\) \(v\) rows of \(i\) beads \\
The \(3\)rd bead is coloured; after the replacement, there are more gaps preceding it.
}
\end{subfigure}
    \caption{
    The replacements \Cref{item:replacing_gaps} and \Cref{item:replacing_beads} of \Cref{lemma:replacing_rows} with \(p=7\), \(i=2\), \(u=3\) and \(v=5\).
    The coloured beads and gaps are those compared in the proof of \Cref{lemma:replacing_rows} when \(j=3\).
    }
    \label{fig:abaci_replacements}
\end{figure}

\begin{proof}
Recall each bead corresponds to a part of size the number of gaps preceding it (that is, the number of gaps in a numerically lower position).
Thus the contribution to the size from a given set of consecutive rows can be broken down into:
\begin{enumerate}[(1)]
    \item\label{item:contribution_gaps}
    gaps in the given rows contributing to parts corresponding to beads after the given rows; and
    \item\label{item:contribution_beads}
    beads in the given rows, whose corresponding parts have contributions from
    \begin{enumerate}[(2.1)]
        \item\label{item:contribution_beads_gaps_below}
        gaps preceding the given rows; and
        \item\label{item:contribution_beads_gaps_within}
        gaps within the given rows.
    \end{enumerate}
\end{enumerate}

The total contribution of \Cref{item:contribution_gaps} is proportional to the number of gaps within the given rows, and the total contribution of \Cref{item:contribution_beads_gaps_below} is proportional to the number of beads within the given rows.
Replacements \Cref{item:replacing_gaps} and \Cref{item:replacing_beads} increase both the number of gaps and beads (at least one of which strictly):
the number of beads under replacement \Cref{item:replacing_gaps} and the number of gaps under replacement \Cref{item:replacing_beads} increase by \(v(p-i) - u(p-(i{+}1))\) which is positive since \(v > u\); 
the number of gaps under replacement \Cref{item:replacing_gaps} and the number of beads under replacement \Cref{item:replacing_beads} increase by \(vi - u(i{+}1)\) which is nonnegative by assumption.

It remains to show that the contribution from \Cref{item:contribution_beads_gaps_within} increases (weakly -- for the remainder of the proof, all comparatives should be considered to be weak).
For this, the rest of the abacus is irrelevant:
we may assume that the abacus consists solely of the rows in question.
An explicit calculation of the size of the partitions corresponding to these abaci is possible; here we present comparisons (illustrated in \Cref{fig:abaci_replacements}) between the abaci which demonstrate the required increase.

[\Cref{item:replacing_beads}]
It suffices to show, for \(1 \leq j \leq u(i{+}1)\), that the \(j\)th bead has more gaps preceding it among \(v\) rows of \(i\) beads than among \(u\) rows of \(i{+}1\) beads (our range considers all the beads in \(u\) rows of \(i{+}1\) beads; there may be additional beads in \(v\) rows of \(i\) beads which further increases the size of that partition).
Among \(v\) rows of \(i\) beads, there are fewer beads in each row, so the \(j\)th bead is in a row further down the abacus (that is, a numerically higher row).
Also, there are more gaps in each row.
Thus, among \(v\) rows of \(i\) beads, the \(j\)th bead has more rows preceding it, each of which has more gaps; thus the \(j\)th bead has more gaps preceding it.

[\Cref{item:replacing_gaps}]
Note that summing the number of gaps preceding each bead is equivalent to summing the number of beads following each gap.
Thus it suffices to show, for \(1 \leq j \leq u(i{+}1)\), that the \(j\)th gap \emph{counted from the end} has more beads following it among \(v\) rows of \(i\) gaps than among \(u\) rows of \(i{+}1\) gaps.
The proof of this is identical to the proof for \Cref{item:replacing_beads} above, after interchanging ``gap'' and ``bead'', and replacing ``down'' with ``up'' (and ``higher'' with ``lower''), and ``preceding'' with ``following''.
(Alternatively, note that an abacus consisting of \(x\) rows of \(j\) gaps and an abacus consisting of \(x\) rows of \(j\) beads are conjugate to each other, in the sense that their corresponding Young diagram are reflections in the main diagonal; conjugation preserves the size of a partition, and thus the claims regarding the contribution from \Cref{item:contribution_beads_gaps_within} are equivalent for \Cref{item:replacing_gaps} and \Cref{item:replacing_beads}.)
\end{proof}

\begin{theorem}
\label{thm:largest_partition}
Let \(p\) be an odd prime.
There is a unique largest \(p\)-core \(p'\)-partition, and it corresponds to the unique longest walk on \(\Gcal_p\).
\end{theorem}

\begin{proof}
By \Cref{prop:walk_hits_p-1_with_every_i}\Cref{item:p-1-recurrent_partition}, the walk corresponding to a largest \(p\)-core \(p'\)-partition is \((p{-}1)\)-recurrent.
Thus it suffices to show that the longest \((p{-}1)\)-recurrent walk maximises the size of the corresponding partition (\Cref{thm:longest_walk} says that this walk is longest amongst all walks on \(\Gcal_p\), and allows us to deduce uniqueness).

Fix \(1 \leq i \leq p{-}1\), and consider the size contributed by the rows corresponding to the closed walk on \(i\)-edges and \((i{+}1)\)-edges between two visits to \(p{-}1\) (these rows have \(i\) and \(i{+}1\) gaps respectively).
Let \(x\) and \(y\) be the numbers of \(i\)- and \((i{+}1)\)-edges taken on this closed walk, where \(0 \leq x,y \leq p{-}1\).
Suppose there exists a longer closed walk which takes \(x'\) steps on \(i\)-edges and \(y'\) steps on \((i{+}1)\)-edges, with \(0 \leq x', y' \leq p-1\) and \(x' + y' > x + y\).
If both \(x' \geq x\) and \(y' \geq y\), then clearly replacing the old closed walk with the new increases the size of the partition (there would be more rows of \(i\) gaps and more rows of \(i{+}1\) gaps).

Suppose instead \(x' > x\) and \(y' < y\).
Let \(u = y - y'\) and \(v = x'-x\); then replacing the old closed walk with the new corresponds precisely to the replacement \Cref{item:replacing_gaps} of \Cref{lemma:replacing_rows}.
Indeed the hypotheses of \Cref{lemma:replacing_rows} hold:
we have \(0 \leq u < v \leq p{-}1\) from the definitions of \(x,x',y,y'\);
it follows that \(vi - u(i{+}1) = (v-u)i - u > -p\);
the closed walks start and finish on the same residue, from which we deduce \(vi - u(i{+}1) \equiv 0 \pmod{p}\);
combining the second inequality with the congruence yields \(vi - u(i{+}1) \geq 0\).
Thus \Cref{lemma:replacing_rows} tells us that taking the longer closed walk corresponds to a larger partition.

Similarly, supposing instead \(y' > y\) and \(x' < x\), setting \(u = x - x'\) and \(v = y'-y\) and using \Cref{lemma:replacing_rows}\Cref{item:replacing_beads} (with \(p{-}1-i\) in place of \(i\)) tells us that taking the longer closed walk corresponds to a larger partition. 
This completes the proof.
\end{proof}

\begin{corollary}
\label{cor:row_multiplicity_symmetry}
The unique largest \(p\)-core \(p'\)-partition has symmetric row multiplicities: \(m_i = m_{p-i}\) for \(2 \leq i \leq p-2\), and \(m_1 = m_{p-1} + 1\).
\end{corollary}

\begin{proof}
A walk from \(p{-}1\) to itself on \(i\)- and \((i{+}1)\)-edges is precisely the reverse of a walk from \(p{-}1\) to itself on \((p-(i{+}1))\)- and \((p-i)\)-edges.
Thus when the closed walks between visits to \(p{-}1\) are chosen to maximise the number of steps taken, the entire walk is symmetric, in the sense of retracing all its steps (except for the first step from \(0\)) after it has taken its last step on a \(\frac{p-1}{2}\)-edge.
\end{proof}

The characterisation of \Cref{thm:largest_partition} and the \((p{-}1)\)-recurrent property allows us to compute the largest \(p\)-core \(p'\)-partition much quicker than an exhaustive search.
Indeed, checking the lengths of the walks from \(p{-}1\) to \(r\) on \(i\)-edges and from \(r\) to \(p{-}1\) on \((i{+}1)\)-edges, for each \(2 \leq i \leq p{-}2\) and \(1 \leq r \leq p{-}1\), has complexity \(O(p^2)\);
in contrast, checking all partitions with row multiplicities below the bound in the proof of \Cref{prop:lower_bound} has complexity \(O(p^{p-1})\).
The partitions found by this computation for small \(p\) are recorded in \Cref{appendix}.

\section*{Acknowledgements}

The author thanks Alexander Miller, Ken Ono, Liron Speyer and the anonymous referees for their helpful comments, and Matt Fayers for his TikZ code for drawing abaci (to which the author has made minor modifications).

\bibliographystyle{alpha}
\bibliography{references}

\newcommand\MSlabel{}\newcommand\MS[2]{McSO22}
\begin{thebibliography}{\MS22}

\bibitem[JK84]{jameskerber1984reptheory}
Gordon James and Adalbert Kerber.
\newblock {\em The Representation Theory of the Symmetric Group}.
\newblock Encyclopedia of Mathematics and its Applications. Cambridge
  University Press, 1984.

\bibitem[\MS22]{mcspiritono2022chartable}
\MSlabel{Eleanor McSpirit and Ken Ono}.
\newblock Zeros in the character tables of symmetric groups with an
  \(\ell\)-core index.
\newblock {\em Canadian Mathematical Bulletin}, June 2022.
\newblock Published online,
  \href{https://doi.org/10.4153/S0008439522000443}{DOI:10.4153/s0008439522000443}.

\end{thebibliography}

\appendix

\section{Largest \texorpdfstring{\(p\)}{p}-core \texorpdfstring{\(p'\)}{p'}-partitions and their sizes for small \texorpdfstring{\(p\)}{p}}  
\label{appendix}

\begin{table}[ht]
    \centering
    \caption{
    The largest \(p\)-core \(p'\)-partitions for \(p \leq 43\),
    described by their row multiplicities,
    found by assessing for each \(i\) the optimal choice of the residue \(r_i\) on which to transition from traversing \(i\)-edges to traversing \((i{+}1)\)-edges in the corresponding walk on \(\Gcal_p\).
    Since the row multiplicities are symmetric (\Cref{cor:row_multiplicity_symmetry}), the second half of the tuple can be deduced from the first, and so is omitted for larger \(p\) (a semicolon denotes the midpoint).
    }
\begin{tabular}[t]{cl} \toprule
$p$ & \multicolumn{1}{l}{Largest $p$-core $p'$-partition (row multiplicities)} \\ \midrule
$ 3 $ & $ (2 ; 1) $ \\
$ 5 $ & $ (4, 2 ; 2, 3) $ \\
$ 7 $ & $ (6, 2, 5 ; 5, 2, 5) $ \\
$ 11 $ & $ (10, 5, 7, 6, 8 ; 8, 6, 7, 5, 9) $ \\
$ 13 $ & $ (12, 5, 7, 10, 8, 11 ; 11, 8, 10, 7, 5, 11) $ \\
$ 17 $ & $ (16, 8, 9, 12, 15, 14, 8, 14 ; 14, 8, 14, 15, 12, 9, 8, 15) $ \\
$ 19 $ & $ (18, 8, 13, 14, 12, 14, 16, 10, 17 ; 17, 10, 16, 14, 12, 14, 13, 8, 17) $ \\
$ 23 $ & $ (22, 11, 15, 16, 19, 17, 18, 18, 14, 15, 20 ; 20, 15, 14, 18, 18, 17, 19, 16, 15, 11, 21) $ \\
$ 29 $ & $ ( 28, 14, 17, 23, 23, 20, 22, 23, 26, 26, 20, 22, 17, 26 ; \ldots ) $ \\
$ 31 $ & $ ( 30, 14, 21, 20, 23, 26, 29, 27, 18, 26, 27, 23, 24, 19, 29 ; \ldots ) $ \\
$ 37 $ & $ ( 36, 17, 23, 27, 27, 30, 33, 28, 32, 33, 24, 34, 35, 29, 32, 26, 21, 35 ; \ldots ) $ \\
$ 41 $ & $ ( 40, 20, 25, 29, 35, 34, 30, 32, 32, 37, 38, 29, 37, 38, 26, 35, 36, 30, 26, 38 ; \ldots ) $ \\
$ 43 $ & $ ( 42, 20, 29, 31, 31, 34, 37, 40, 36, 37, 37, 30, 31, 38, 38, 29, 38, 38, 32, 28, 41 ; \ldots ) $ \\
\bottomrule
\end{tabular}

    \label{tab:partitions}
\end{table}

\begin{table}[ht]
    \centering
    \caption{
    Sizes of the largest \(p\)-core \(p'\)-partitions for \(p \leq 43\), as well as the sizes of the large \(p\)-core \(p'\)-partitions explicitly described in \Cref{prop:lower_bound} and the upper bound on the size of \(p\)-core \(p'\)-partitions from \Cref{prop:upper_bound}.
    }
    \label{tab:sizes}
\begin{tabular}[t]{cS[table-format=9]S[table-format=9]S[table-format=9]} \toprule
$p$ & \text{Explicit partition} & \text{Largest partition} & \text{Upper bound} \\ \midrule
$ 3 $ & 10 & 10 & 10 \\
$ 5 $ & 187 & 198 & 289 \\
$ 7 $ & 1326 & 1726 & 2701 \\
$ 11 $ & 19134 & 29773 & 50500 \\
$ 13 $ & 51655 & 93334 & 146015 \\
$ 17 $ & 255671 & 502140 & 788476 \\
$ 19 $ & 496802 & 1006386 & 1577550 \\
$ 23 $ & 1556950 & 3312177 & 5158945 \\
$ 29 $ & 6234927 & 14508172 & 21523915 \\
$ 31 $ & 9295954 & 22313239 & 32413475 \\
$ 37 $ & 26832011 & 68032781 & 95761401 \\
$ 41 $ & 49641139 & 127172362 & 179231950 \\
$ 43 $ & 66042990 & 171105947 & 239637580 \\
\bottomrule
\end{tabular}

\end{table}

\end{document}